\documentclass{aptpub}
\usepackage[polish,UKenglish]{babel}
\usepackage{enumerate}
\usepackage{url}
\authornames{A. J. OSTASZEWSKI}
\shorttitle{Stable laws and Beurling kernels}

\newcommand{\oh}{\mathrm o}
\newcommand{\Oh}{\mathrm O}
\providecommand{\abs}[1]{\lvert #1\rvert }
\newtheorem{ThmGFE}{\noindent Theorem GFE}

\begin{document}

\title{Stable laws and Beurling kernels}

\author[London School of Economics]{Adam J. Ostaszewski}
\address{Mathematics Department, London School of Economics,
 Houghton Street, London WC2A 2AE, UK}
\email{A.J.Ostaszewski@lse.ac.uk}

\begin{abstract}We identify a close relation between stable
distributions and the limiting homomorphisms central to the theory of
regular variation. In so doing some simplifications are achieved in the
direct analysis of these laws in Pitman and Pitman (2016); stable
distributions are themselves linked to homomorphy.
\end{abstract}



\renewcommand{\thefootnote}{\number \count63}% restoring LaTeX default
\setcounter{footnote}{0}%

\section{Introduction}\label{s:intro}

This note\footnote{This expanded version of \cite{OstA} includes new material in \S 4 and an Appendix.} takes its inspiration from Pitman and Pitman's approach \cite{PitP},
in this volume, to the characterization of stable laws \emph{directly} from
their characteristic functional equation \cite[(2.2)]{PitP}, \eqref{ChFE}
below, which they complement with the derivation of parameter restrictions by
an appeal to \emph{Karamata} (classical) regular variation (rather than
\emph{indirectly} as a special case of the L\'{e}vy--Khintchine
characterization of infinitely decomposable laws---cf.\ \cite[Section
  4]{PitP}). We take up their functional-equation tactic with three aims in
mind.  The first and primary one is to extract a hidden connection with the
more general theory of \emph{Beurling regular variation}, which embraces the
original Karamata theory and its later `Bojani\'{c}--Karamata--de Haan'
variants. (This has received renewed attention: \cite{BinO1,BinO4,Ost1}).  The
connection is made via another functional equation, the \emph{Goldie equation}
\begin{equation}\label{GFE}
\kappa(x+y)-\kappa(x)=\gamma(x)\kappa(y)\qquad(x,y\in\mathbb{R}),\tag{\emph{GFE}}
\end{equation}
with \emph{vanishing side condition} $\kappa(0)=0$ and \emph{auxiliary
function }$\gamma$, or more properly with its multiplicative variant:
\begin{equation}\label{GFEx}
K(st)-K(s)=G(s)K(t)\qquad (s,t\in \mathbb{R}_+:=(0,\infty)),\tag{${\it GFE}_{\times}$}
\end{equation}
with corresponding side condition $K(1)=0$; the additive variant arises first
in \cite{BinG} (see also \cite[Lemma 3.2.1 and Theorem 3.2.5]{BinGT}), but has
only latterly been so named in recognition of its key role both there and in
the recent developments \cite{BinO2,BinO3}, inspired both by \emph{Beurling
slow variation} \cite[Section 2.11]{BinGT} and by its generalizations
\cite{BinO1,BinO4} and \cite{Ost1}. This equation describes the family of
\emph{Beurling kernels} (the asymptotic homomorphisms of Beurling regular
variation), that is, the functions $K_{F}$ arising as locally uniform limits
of the form
\begin{equation}\label{BKer}
K_{F}(t):=\lim_{x\rightarrow \infty }[F(x+t\varphi (x))-F(x)],
\tag{\emph{BKer}}
\end{equation}
for $\varphi(\cdot)$ ranging over \emph{self-neglecting} functions (\emph{SN}).
(See \cite{Ost1,Ost2} for the larger family of kernels arising when $\varphi(\cdot)$
ranges over the \emph{self-equivarying} functions (\emph{SE}), both classes
recalled in the complements section \ref{ss:SNSE}.)

A secondary aim is achieved in the omission of extensive special-case
arguments for the limiting cases in the Pitman analysis (especially the case
of characteristic exponent $\alpha =1$ in \cite[Section 5.2]{PitP}---affecting
parts of \cite[Section 8]{PitP}), employing here instead the more natural approach of
  interpreting the `generic' case `in the limit' via the l'Hospital rule. A
  final general objective of further streamlining is achieved, \emph{en
    passant}, by telescoping various cases into one, simple, group-theoretic
  argument; this helps clarify the `group' aspects as distinct from
  `asymptotics', which relate parameter restrictions to tail balance---see
  Remark \ref{r:dominant}.

A random variable $X$ has a \emph{stable law} if for each $n\in \mathbb{N}$
the law of the random walk $S_{n}:=X_{1}+\dotsb+X_{n},$ where the $n$ steps are
independent and with law identical to $X$, is of the same type, i.e.\ the
same in distribution up to scale and location:
\[
S_{n}\eqdist a_{n}X+b_{n},
\]
for some real constants $a_{n}$, $b_{n}$ with $a_{n}>0$; cf.\ \cite[VI.1]{Fel} and
\cite[(1.1)]{PitP}. These laws may be characterized by the \emph{characteristic
functional equation} (of the characteristic function of $X$, $\varphi (t)=
\mathbb{E}[\re^{\ri tX}]$), as in \cite[(2.2)]{PitP}:
\begin{equation}\label{ChFE}
\varphi(t)^n=\varphi(a_nt)\exp(\ri b_nt)\qquad(n\in\mathbb{N},\;
  t\in\mathbb{R}_+).\tag{\emph{ChFE}}
\end{equation}

The standard way of solving \eqref{ChFE} is to deduce the equations satisfied by
the functions $a:n\mapsto a_{n}$ and $b:n\mapsto b_{n}$. Pitman and Pitman
\cite{PitP} proceed directly by proving the map $a$ \emph{injective}, then
extending the map $b$ to $\mathbb{R}_{+}:=(0,\infty )$, and exploiting the
classical Cauchy (or Hamel) exponential functional equation (for which see
\cite{AczD} and \cite{Kuc}):
\begin{equation}\label{CEE}
K(xy)=K(x)K(y)\qquad (x,y\in \mathbb{R}_{+});\tag{\emph{CEE}}
\end{equation}
\eqref{CEE} is satisfied by $K(\cdot)=a(\cdot)$ on the smaller domain
$\mathbb{N}$, as a consequence of \eqref{ChFE}. See \cite{RamL} for a similar,
but less self-contained account. For other applications see the recent
\cite{GupJTS}, which characterizes `generalized stable laws'.

We show in Section \ref{s:reduction} the surprising equivalence of
\eqref{ChFE} with the fundamental equation \eqref{GFE} of the recently
established theory of \emph{Beurling regular variation}. There is thus a
one-to-one relation between Beurling kernels arising through \eqref{BKer} and
the continuous solutions of \eqref{ChFE}, amongst which are the one-dimensional stable
distributions. This involves passage from discrete to continuous, a normal
feature of the theory of regular variation (see \cite[Section 1.9]{BinGT})
which, rather than unquestioningly adopt, we track carefully via Lemma 1 and
Corollary 1 of Section \ref{s:reduction}: the ultimate justification here is the
extension of $a$ to $\mathbb{R}_{+}$ (Ger's extension theorem \cite[Section
  18.7]{Kuc} being thematic here), and the continuity of characteristic
functions.

The emergence of a particular kind of functional equation, one interpretable
as a \emph{group} homomorphism (see Section \ref{ss:Homo}), is linked to the
simpler than usual form here of `probabilistic associativity' (as in
\cite{Bin}) in the incrementation process of the stable random walk; in more
general walks, functional equations (and integrated functional equations---see
\cite{RamL}) arise over an associated \emph{hypergroup}, as with the
Kingman--Bessel hypergroup and Bingham-Gegenbauer (ultraspherical) hypergroup (see
\cite{Bin} and \cite{BloH}). We return to these matters, and connections with
the theory of flows, elsewhere---\cite{Ost3}.

The material is organized as follows. Below we identify the solutions to
\eqref{GFE} and in Section \ref{s:reduction} we prove equivalence of
\eqref{GFE} and \eqref{ChFE}; our proof is self-contained modulo the
(elementary) result that, for $\varphi $ a characteristic function,
\eqref{ChFE} implies $a_{n}=n^{k}$ for some $k>0$ (in fact we need only to know that $k\neq 0$). Then in Section \ref{s:form} we read off the form of the characteristic functions of the
stable laws. In  Section \ref{s:sequenceidentification} we show that, for an arbitrary continuous solution $\varphi$ of \eqref{ChFE}, necessarily $a_{n}=n^{k}$ for some $k\neq 0$.
We conclude in Section \ref{s:complements} with complements
describing the families \emph{SN} and \emph{SE} mentioned above, and
identifying the group structure implied, or `encoded', by \eqref{GFEx} to be
$(\mathbb{R}_{+},\times )$, the multiplicative positive reals. In the Appendix we offer an elementary derivation of a a key formula needed in \cite{PitP}.

The following result, which has antecedents in several settings (some cited
below), is key; on account of its significance, this has recently received
further attention in \cite[especially Theorem 3]{BinO2} and \cite[especially
Theorem 1]{Ost2}, to which we refer for background---cf.\ Section
\ref{ss:ThmGFE}.

% Local redefinition for an un-numbered theorem:
{%
\makeatletter
\def\th@plain{\normalfont\itshape
  \def\@begintheorem##1##2{%
        \item[\hskip\labelsep \theorem@headerfont ##1{\bf .}] }%
}%
\makeatother

\begin{ThmGFE} {\rm (\cite[Theorem 1]{BinO2}, \cite[(2.2)]{BojK},
  \cite[Lemma 3.2.1]{BinGT}; cf.\ \cite{AczG}.)}
For $\mathbb{C}$-valued functions $\kappa$ and $\gamma$ with $\gamma$ locally
bounded at $0$, with $\gamma(0)=1$ and $\gamma\neq1$
except at $0$, if $\kappa\not\equiv0$ satisfies
\eqref{GFE} subject to the side condition $\kappa(0)=0$, then for some
$\gamma_0$, $\kappa_0\in \mathbb{C}$:
\[
\gamma(u)=\re^{\gamma_0u}\quad\text{and}\quad
\kappa(x)\equiv\kappa_0H_{\gamma_0}(x):=\kappa_0\int_0^x\gamma(u)\sd u
=\kappa_0\frac{\re^{\gamma_0x}-1}{\gamma_0},
\]
under the usual l'Hospital convention for interpreting $\gamma_0=0$.
\end{ThmGFE}
}% end local redefinition

\begin{rem}\label{r:extend}
The cited proof is ostensibly for
$\mathbb{R}$-valued $\kappa(\cdot)$ but immediately extends to
$\mathbb{C}$-valued $\kappa$. Indeed, in brief, the proof rests on symmetry:
\[
\gamma (v)\kappa(u)+\kappa(v)=\kappa(u+v)=\kappa(v+u)
=\gamma(u)\kappa(v)+\kappa(u).
\]
So, for $u$, $v$ not in $\{x:\gamma(x)=1\}$, an additive subgroup,
\[
\kappa(u)[\gamma(v)-1]=\kappa(v)[\gamma(u)-1]:\qquad
\frac{\kappa(u)}{\gamma(u)-1}=\frac{\kappa(v)}{\gamma(v)-1}
=\kappa_0,
\]
as in \cite[Lemma 3.2.1]{BinGT}. If $\kappa(\cdot)$ is to satisfy \eqref{GFE},
$\gamma(\cdot)$ needs to satisfy \eqref{CEE}.

The notation $H_\rho$ (originating in \cite{BojK}) is from \cite[Chapter 3: de
  Haan theory]{BinGT} and, modulo exponentiation, links to the `inverse'
functions $\eta_\rho(t)=1+\rho t$ (see Section \ref{ss:Homo}) which permeate
regular variation (albeit long undetected), a testament to the underlying
\emph{flow} and \emph{group} structure, for which see especially
\cite{BinO1,BinO4}.

The Goldie equation is a special case of the \emph{Levi--Civita
equations}; for a text-book treatment of their solutions for domain a
semigroup and range $\mathbb{C}$ see \cite[Chapter 5]{Ste}.
\end{rem}

\begin{rem}\label{r:constants}
We denote the constants $\gamma_0$ and $\kappa_0$ more
simply by $\gamma$ and $\kappa$, whenever context permits. To prevent
conflict with the $\gamma$ of \cite[Section 5.1]{PitP} we denote that
here by $\gamma_{\text{P}}(k),$ showing also dependence on the index of
growth of $a_n$: see Section \ref{ss:notation}.
\end{rem}

\begin{rem}\label{r:stuv}
To solve \eqref{GFEx} write $s=\re^u$ and $t=\re^v$,
obtaining \eqref{GFE}; then
\begin{align*}
G(\re^u)&=\gamma(u)=\re^{\gamma u}:\qquad G(s)=s^{\gamma}\\
K(\re^u)&=\kappa(u)=\kappa\,\frac{\re^{\gamma u}-1}\gamma:\qquad
  K(s)=\kappa\,\frac{s^\gamma-1}\gamma.
\end{align*}
\end{rem}

\begin{rem}\label{r:altreg}
Alternative regularity conditions, yielding
continuity and the same $H_\gamma$ conclusion, include in \cite[Theorem
  2]{BinO2} the case of $\mathbb{R}$-valued functions with $\kappa(\cdot)$ and
$\gamma(\cdot)$ both non-negative on $\mathbb{R}_+$ with $\gamma\neq1$ except
at $0$ (as then either $\kappa\equiv0$, or both are continuous).
\end{rem}

\section{Reduction to the Goldie Equation}\label{s:reduction}

In this section we establish a Proposition connecting \eqref{ChFE} with
\eqref{GFEx}, and so stable laws with Beurling kernels. Here in the interests of
brevity\footnote{In \S 4 we prove from \eqref{ChFE}, with $\varphi$ arbitrary
but continuous, that $a_n=n^k$ for some $k\ne0$, cf. \cite{Ost3}.}, this
makes use of a well-known result concerning the norming constants
(cf.\ \cite[VI.1, Theorem 1]{Fel}, \cite[Lemma 5.3]{PitP}), that $a:n\mapsto
a_n$ satisfies $a_n=n^k$ for some $k>0$, and so is extendible to a continuous
surjection onto $\mathbb{R}_+:=(0,\infty)$:
\[
\tilde a(\nu)=\nu^k\qquad(\nu>0);
\]
this is used below to justify the validity of the definition
\[
f(t):=\log\varphi(t)\qquad(t>0),
\]
with $\log$ here the principal logarithm, a tacit step in \cite[Section 5.1]{PitP},
albeit based on \cite[Lemma 5.2]{PitP}. We write
$a_{m/n}=\tilde a_{m/n}=a_m/a_n$ and put
$\mathbb{A}_{\mathbb{N}}:=\{a_n:n\in\mathbb{N}\}$ and
$\mathbb{A}_{\mathbb{Q}}:=\{a_{m/n}:m,n\in \mathbb{N}\}$.

The Lemma below re-proves an assertion from \cite[Lemma 5.2]{PitP}, but without
assuming that $\varphi$ is a characteristic function. Its Corollary needs
no explicit formula for $b_{m/n},$ since the term will eventually be
eliminated.

% Local redefinition for un-numbered proclaims:
{%
%\makeatletter
%\def\th@plain{\normalfont\itshape
  %\def\@begintheorem##1##2{%
 %       \item[\hskip\labelsep \theorem@headerfont ##1{\bf .}] }%
%}%
%\makeatother

\begin{lemma}\label{l}
For continuous $\varphi\not\equiv0$ satisfying \eqref{ChFE} with $a_n=n^k$
($k\ne0$), $\varphi$ has no zeros on $\mathbb{R}_+$.
\end{lemma}

\begin{proof}If $\varphi(\tau)=0$ for some $\tau>0$ then
$\varphi(a_m\tau)=0$ for all $m$, by \eqref{ChFE}. Again by \eqref{ChFE},
$\abs{\varphi(\tau a_m/a_n)}^n=\abs{\varphi(a_m\tau)}=0$, so $\varphi$ is
zero on the dense subset of points $\tau a_m/a_n$; then, by continuity,
$\varphi\equiv0$ on $\mathbb{R}_+$, a contradiction.
\end{proof}

\begin{corollary}\label{c}
The equation \eqref{ChFE} with continuous $\varphi\not\equiv0$ and $a_n=n^k$ ($k\ne0$) holds on the dense
subgroup $\mathbb{A}_{\mathbb{Q}}$: there are constants
$\{b_{m/n}\}_{m,n\in\mathbb{N}}$ with
$$\varphi(t)^{m/n}=\varphi(a_{m/n}t)\exp(\ri b_{m/n}t)\qquad(t\ge0).$$
\end{corollary}

\begin{proof}Taking $t/a_n$ for $t$ in \eqref{ChFE} gives
$\varphi(t/a_n)^n=\varphi(t)\exp(\ri b_nt/a_n)$,
so by Lemma 1, using principal values,
$\varphi(t)^{1/n}=\varphi(t/a_n)\exp(-\ri tb_n/(na_n))$, whence
$$\varphi(t)^{m/n}=\varphi\Bigl(\frac t{a_n}\Bigr)^m
  \exp\Bigl(-\frac{\ri tmb_n}{na_n}\Bigr).$$
Replacing $n$ by $m$ in \eqref{ChFE} and then replacing $t$ by
$t/a_n$ gives $\varphi(t/a_n)^m=\varphi(a_mt/a_n)\break\exp(\ri
b_mt/a_n)$. Substituting this into the above and using $a_m/a_n=a_{m/n}$:
$$\varphi(t)^{m/n}=\varphi(a_{m/n}t)\exp\Bigl(\ri t\,\frac{nb_m-mb_n}{na_n}\Bigr).$$
As the left-hand side, and the first term on the right, depend on $m$ and $n$
only through $m/n$, we may rewrite the constant $(nb_m-mb_n)/(na_n)$
as $b_{m/n}$. The result follows.
\end{proof}
\bigskip

Our main result below, on equational equivalence, uses a condition
\eqref{GARplus}
applied to the dense subgroup
$\mathbb{A}=\mathbb{A}_{\mathbb{Q}}$.
This is a
\emph{quantifier weakening} relative
to \eqref{GFE} and is similar to a condition with all variables ranging over
$\mathbb{A}=\mathbb{A}_{\mathbb{Q}}$, denoted (${\it G}_{\mathbb{A}}$) in \cite{BinO2},
to which we refer for background on quantifier weakening. In Proposition 1
below we may also impose just $({\it G}_{\mathbb{A}_{\mathbb{Q}}})$, granted
continuity of $\varphi$.

\begin{proposition}\label{p}
For $\varphi$ continuous and $a_n=n^k$ ($k\ne0$), the functional equation
\eqref{ChFE} is equivalent to
\begin{equation}\label{GARplus}
K(st)-K(s)=K(t)G(s)\qquad(s\in\mathbb{A},\;t\in\mathbb{R}_+),
\tag{${\it G}_{\mathbb{A},\mathbb{R}_+}$}
\end{equation}
for either of $\mathbb{A}=\mathbb{A}_{\mathbb{N}}$ or
$\mathbb{A}=\mathbb{A}_{\mathbb{Q}}$, both with side condition $K(0)=1$ and
with $K$ and $G$ continuous; the latter directly implies \eqref{GFEx}. The
correspondence is given by
\[
K(t)=
\begin{cases}
  \displaystyle \frac{f(t)}{t\mathstrut},&\text{if $f(1)=0$},\\
  \displaystyle \frac{f(t)}{tf(1)}-1,&\text{if $f(1)\neq0$}.
\end{cases}
\]
\end{proposition}
}% end local redefinition

\begin{proof}By the Lemma, using principal values, \eqref{ChFE} may
be re-written as
\[
\varphi(t)^{n/t}=\varphi(a_nt)^{1/t}\exp(\ri b_n)\qquad
(n\in\mathbb{N},\;t\in \mathbb{R}_+).
\]
From here, on taking principal logarithms and adjusting notation ($f:=\log
\varphi $, $h(n)=-$\textrm{i}$b_{n},$ and $g(n):=a_{n}\in \mathbb{R}_{+}$),
pass first to the form%
\[
\frac{f(g(n)t)}t=\frac{nf(t)}t+h(n)\qquad(n\in\mathbb{N},\;t\in\mathbb{R}_+);
\]
here the last term does not depend on $t$, and is defined for each $n$ so as
to achieve equality. Then, with $s:=g(n)\in\mathbb{R}_{+}$, replacement of
$n$ by $g^{-1}(s)$, valid by injectivity, gives, on cross-multiplying by $t$,
\[
f(st)=g^{-1}(s)f(t)+h(g^{-1}(s))t.
\]
As $s,t\in \mathbb{R}_+$, take $F(t):=f(t)/t$, $G(s):=g^{-1}(s)/s$,
$H(s):=h(g^{-1}(s))/s$; then
\begin{equation}\label{dag}
F(st)=F(t)G(s)+H(s)\qquad (s\in\mathbb{A}_{\mathbb{N}},\;t\in\mathbb{R}_+).
\tag{\dag }
\end{equation}
This equation contains \emph{three} unknown functions: $F$, $G$, $H$ (cf.\ the
Pexider-like formats considered in \cite[Section 4]{BinO2}), but we may reduce
the number of unknown functions to \emph{two} by entirely
eliminating\footnote{This loses the ``affine action'': $K\mapsto
  G(t)K+H(t)$.} $H$. The elimination argument splits according as $F(1)=f(1)$
is zero or not.

\begin{enumerate}[\it C{a}se 1:\/]
\item $f(1)=0$ (i.e.\ $\varphi (1)=1$). Taking $t=1$ in \eqref{dag} yields
  $F(s)=H(s)$, and so \eqref{GARplus} holds for $K=F$, with side condition
  $K(1)=0$ ($=F(1)$).

\item $f(1)\neq 0$. Then, with $\tilde{F}:=F/F(1)$ and
$\tilde{H}:=H/F(1)$ in \eqref{dag},
\[
\tilde{F}(st)=\tilde{F}(t)G(s)+\tilde{H}(s)\qquad
(s\in\mathbb{A},\;t\in\mathbb{R}_+),
\]
and $\tilde{F}(1)=1$. Taking again $t=1$ gives
$\tilde{F}(s)=G(s)+\tilde{H}(s)$. Setting
\begin{equation}\label{dagdag}
K(t):=\tilde{F}(t)-1=\frac{F(t)}{F(1)}-1\tag{\dag\dag }
\end{equation}
(so that $K(1)=0$), and using $\tilde{H}=\tilde{F}-G$ in \eqref{dag} gives
\begin{align*}
\tilde{F}(st)&=\tilde{F}(t)G(s)+\tilde{F}(s)-G(s),\\
(\tilde{F}(st)-1)-(\tilde{F}(s)-1)&=(\tilde{F}(t)-1)G(s),\\
K(st)-K(s)&=K(t)G(s).
\end{align*}
That is, $K$ satisfies \eqref{GARplus} with side condition $K(1)=0$.
\end{enumerate}

In summary: in both cases elimination of $H$ yields
$(G_{\mathbb{A},\mathbb{R}_+})$ and the side condition of vanishing at the identity.

So far, in \eqref{GARplus} above, $t$ ranges over $\mathbb{R}_+$ whereas $s$
ranges over $\mathbb{A}_{\mathbb{N}}=\{a_n:n\in\mathbb{N}\}$, but $s$ can
be allowed to range over $\{a_{m/n}:m,n\in \mathbb{N}\}$, by the
Corollary. As before, since $a:n\mapsto a_n$ has $\tilde{a}$ as its
continuous extension to a bijection onto $\mathbb{R}_+$, and $\varphi$
is continuous, we conclude that $s$ may range over
$\mathbb{R}_+$, yielding the multiplicative form of the Goldie equation
\eqref{GFEx} with the side-condition of vanishing at the identity.
\end{proof}

\begin{rem}\label{r:onecase}
As in \cite[Section 5]{PitP}, we consider only
\emph{non-degenerate} stable distributions, consequently `Case 1' will not figure
below (as this case yields an arithmetic distribution---cf.\ \cite[XVI.1,
  Lemma 4]{Fel}, so here concentrated on $0$).
\end{rem}

\begin{rem}\label{r:case2}
In `Case 2' above,
$\tilde{H}(st)-\tilde{H}(s)=\tilde{H}(t)G(s)$,
since $G(st)=G(s)G(t)$, by Remark \ref{r:altreg}. So
$\tilde H(\re^u)=\kappa H_{\gamma}(u)=\kappa(\re^{\gamma u}-1)/\gamma$. We use
this in Section \ref{s:form}.
\end{rem}

\section{Stable laws: their form}\label{s:form}

This section demonstrates how to `telescope' several cases of the analysis in
\cite{PitP} into one, and to make l'Hospital's Rule carry the burden of the
`limiting' case $\alpha =1$. At little cost, we also deduce the form of the
location constants $b_n$, without needing the separate analysis conducted in
\cite[Section 5.2]{PitP}.

We break up the material into steps, beginning with a statement of the
result.

\subsection{Form of the law}\label{ss:form}

The \emph{form} of $\varphi $ for a non-degenerate stable distribution is an
immediate corollary of Theorem GFE (Section \ref{s:intro}) applied to
\eqref{dagdag} above. For some $\gamma\in\mathbb{R}$, $\kappa\in\mathbb{C}$
and with $A:=\kappa/\gamma$ and $B:=1-A$,
\begin{equation}\label{ddag}
f(t)=\log\varphi(t)=
\begin{cases}
f(1)(At^{\gamma+1}+Bt),&\text{for $\gamma\neq 0$},\\
f(1)(t+\kappa t\log t),&\text{with $\gamma=0$},
\end{cases}
\qquad(t>0).\tag{\ddag}
\end{equation}
Here $\alpha:=\gamma+1$ is the \emph{characteristic exponent}. From this
follows a formula for $t<0$ (by complex conjugation---see below). The
connection with \cite[Section 5 at end]{PitP} is given by:

\begin{enumerate}[(i)]
\item $f(1):=\log\varphi(1)=-c+\ri y$ (with $c>0$, as
$\abs{\varphi(t)}<1$ for some $t>0$);

\item $f(1)\kappa=-\ri\lambda$. So $f(1)B=-c+\ri(y+\lambda/\gamma)$,
and $\kappa=\lambda(-y+\ri c)/(c^2+y^2)$.
\end{enumerate}

\begin{rem}\label{r:dominant}
We note, for the sake of completeness, that
restrictions on the two parameters $\alpha$ and $\kappa$ (equivalently
$\gamma$ and $\kappa$) follow from asymptotic analysis of the `initial'
behaviour of the characteristic function $\varphi$ (i.e.\ near the origin).
This is equivalent to the `final' or tail behaviour (i.e.\ at infinity) of
the corresponding distribution function. Specifically, the `dominance ratio'
of the imaginary part of the \emph{dominant} behaviour in $f(t)$ to the
value $c$ (as in (i) above) relates to the `tail balance' ratio $\beta$ of
\cite[(6.10)]{PitP}, i.e.\ the asymptotic ratio of the distribution's tail
difference to its tail sum---cf.\ \cite[Section 8]{PitP}. Technical arguments, based
on Fourier inversion, exploit the regularly varying behaviour as
$t\downarrow 0$ (with index of variation $\alpha $---see above) in the real
and imaginary parts of $1-\varphi(t)$ to yield the not unexpected result
\cite[Theorem 6.2]{PitP} that, for $\alpha\neq1$, the dominance ratio is
proportional to the tail-balance ratio $\beta$ by a factor equal to the
ratio of the sine and cosine variants of Euler's Gamma integral\footnote{%
In view of that factor's key role, a quick and elementary derivation is
offered in the Appendix (for $0<\alpha<1$).} (on account of
the dominant power function)---compare \cite[Theorem 4.10.3]{BinGT}.
\end{rem}

\subsection{On notation}\label{ss:notation}
The parameter $\gamma:=\alpha-1$ is
linked to the auxiliary function $G$ of \eqref{GFE}; this usage of $\gamma$
conflicts with \cite{PitP}, where two letters are used in describing the behaviour of the ratio
$b_n/n$: $\lambda$ for the `case $\alpha=1$', and otherwise $\gamma$
(following Feller \cite[VI.1 Footnote 2]{Fel}). The latter we denote by
$\gamma _{\text{P}}(k)$, reflecting the $k$ value in the `case $\alpha=1/k\neq 1$'.
In Section \ref{ss:locgen} below it emerges that $\gamma_{\text{P}}(1+)=\lambda\log n$.

\subsection{Verification of the form (\ref{ddag})}\label{ss:ver}
By Remark \ref{r:onecase}, only the second case of the Proposition applies:
the function $K(t)=\tilde{F}(t)-1=f(t)/(tf(1))-1$ solves \eqref{GFEx} with
side-condition $K(1)=0$.  Writing $t=e^u$ (as in Remark \ref{r:stuv}) yields
\[
\frac{f(t)}{tf(1)}=\frac{f(\re^u)\re^{-u}}{f(1)}=1+K(\re^u)=\kappa(u)
=1+\kappa\,\frac{\re^{\gamma u}-1}\gamma,
\]
for some complex $\kappa$ and $\gamma\neq0$ (with passage to $\gamma=0$,
in the limit, to follow). So, for $t>0$, with $A:=\kappa/\gamma$ and
$B:=1-A$, as above,
\[
f(t)=\log \varphi(t)=f(1)t\Bigl(1+\kappa\,\frac{t^\gamma-1}\gamma\Bigr)
=f(1)(At^\alpha+Bt),
\]
with $\alpha=\gamma+1$. On the domain $t>0$, this agrees with \cite[(5.5)]{PitP};
for $t<0$ the appropriate formula is immediate via complex conjugation,
verbatim as in the derivation of \cite[(5.5)]{PitP}, save for the $\gamma$
usage. To cover the case $\gamma=0$, apply the l'Hospital convention; as in
\cite[(5.8)]{PitP}, for $t>0$ and $u>0$ and some $\kappa\in\mathbb{C}$,
\[
\kappa (t):=\frac{f(e^t)e^{-t}}{f(1)}=1+\kappa t:\qquad
f(u)=f(1)(u+\kappa u\log u).
\]

\subsection{Location parameters: general case $\alpha\neq1$}\label{ss:locgen}
Here $\gamma =\alpha -1\neq 0$. From the proof of the Proposition,
$G(t):=g^{-1}(\re^t)\re^{-t}$, so $g^{-1}(\re^t)=\re^t\re^{\gamma t}=\re^{\alpha t}$.
Put $k=1/\alpha$; then
\[
v=g^{-1}(u)=u^\alpha:\qquad u=g(v)=v^{1/\alpha}=v^k,
\]
confirming $a_n=g(n)=n^k$, as in \cite[Lemma 5.3]{PitP}. (Here $k>0$, as
\emph{strict} monotonicity was assumed in the Proposition). Furthermore,
as in Remark \ref{r:case2},
$$\kappa\,\frac{\re^{\gamma t}-1}\gamma=\tilde H(\re^t)
=\frac{h(g^{-1}(\re^t))\re^{-t}}{f(1)};$$
so
$$
h(g^{-1}(e^t))=f(1)\kappa\,\frac{\re^{\alpha t}-\re^t}\gamma:\qquad
h(u)=f(1)\kappa\,\frac{u-u^{1/\alpha}}\gamma
=f(1)\kappa\,\frac{u-u^k}\gamma,
$$
where $\gamma=\alpha-1=(1-k)/k$. So
$$b_n=\ri h(n)=\ri f(1)\kappa\,\frac{n-n^k}\gamma,$$
as in the Pitman analysis: see
\cite[Section 5.1]{PitP}. Here $b_n$ is real, since $f(1)\kappa=-\ri\lambda$,
according to (ii) in Section \ref{ss:form} above and conforming with
\cite[Section 5.1]{PitP}. So as $b_n/n=\gamma _{\text{P}}(k)$, similarly to
\cite[end of proof of Lemma 4.1]{PitP}, again as $f(1)\kappa=-\ri\lambda$, for any $n\in \mathbb{N}$
\[
\lim_{k\rightarrow1}\gamma_{\text{P}}(k)=\ri f(1)\kappa\,
\lim_{k\rightarrow1}k\frac{1-n^{k-1}}{k-1}=\lambda\log n.
\]

\subsection{Location parameters: special case $\alpha=1$}\label{ss:special}
Here $\gamma =0$. In Section \ref{ss:ver} above the form of $g$ specializes to
\[
g^{-1}(\re^t)=\re^t:\qquad g(u)=u.
\]
Applying the l'Hospital convention yields the form of $h$: for $t>0$ and
$u>0$,
\[
h(g^{-1}(\re^t))=h(\re^t)=f(1)\kappa t\re^t:\qquad
h(u)=f(1)\kappa u\log u;
\]
so, as in \cite[(5.8)]{PitP}, $b_n=\lambda n\log n$ (since
$b_n=\ri h(n)$ and again $\lambda=\ri f(1)\kappa$).

\section{Identifying $a_{n}$ from the continuity of $\protect\varphi $}\label{s:sequenceidentification}

In \S 3 the form of the continuous solutions $\varphi $ of $(ChFE)$ was
derived from the known continuous solutions of the Goldie equation $(GFE)$
on the assumption that $a_{n}=n^{k}$, for some  $k\neq 0$ (as then $%
\{a_{m}/a_{n}:m,n\in \mathbb{N}\}$ is dense in $\mathbb{R}_{+})$. Here we
show that the side condition on $a_{n}$ may itself be deduced from $(ChFE)
$ provided the solution $\varphi $ is continuous and \textit{non-trivial,}
i.e. neither $|\varphi |\equiv 0$ nor $|\varphi |\equiv 1$
holds, so obviating the assumption that $\varphi $ is the characteristic
function of a (non-degenerate) distribution.

\bigskip
% Local redefinition for un-numbered proclaims:
{%
\makeatletter
\def\th@plain{\normalfont\itshape
  \def\@begintheorem##1##2{%
        \item[\hskip\labelsep \theorem@headerfont ##1{\bf .}] }%
}%
%\makeatother
\begin{theorem}If $\varphi $ is a non-trivial
continuous function and satisfies $(ChFE)$ for some sequence %
$a_{n}\geq 0$, then $a_{n}=n^{k}$ for some $k\neq 0.$
\end{theorem}
}% end local redefinition
\bigskip

We will first need to establish a further lemma and proposition.

\bigskip

\begin{lemma}If $(ChFE)$\ is satisfied by a\
non-trivial continuous function $\varphi$, then the sequence %
$a_{n}$ is either convergent to $0,$ or divergent
(`convergent to $+\infty $').
\end{lemma}

\begin{proof}Suppose otherwise. Then for some $\mathbb{M}%
\subseteq \mathbb{N}$, and $a>0,$%
\[
a_{m}\rightarrow a\text{ through }\mathbb{M}.
\]%
W.l.o.g. $\mathbb{M}=\mathbb{N}$, otherwise interpret $m$ below as
restricted to $\mathbb{M}$. For any $t,$ $a_{m}t\rightarrow at,$ so $%
K_{t}:=\sup_{m}\{|\varphi (a_{m}t)|\}$ is finite. Then for all $m$%
\[
|\varphi (t)|^{m}=|\varphi (a_{m}t)|\leq K_{t},
\]%
and so $|\varphi (t)|\leq 1,$ for all $t.$ By continuity,
\[
|\varphi (at)|=\lim_{m}|\varphi (a_{m}t)|=\lim_{m}|\varphi (t)|^{m}=0\text{
or }1.
\]%
Then, setting $N_{k}:=\{t:|\varphi (at)|=k\},$%
\[
\mathbb{R}_{+}=N_{0}\cup N_{1}.
\]%
By the connectedness of $\mathbb{R}_{+}$, one of $N_{0},N_{1}$ is empty, as
the sets $N_{k}$ are closed; so respectively $|\varphi |\equiv 0$ or $%
|\varphi |\equiv 1,$ contradicting non-triviality. 
\end{proof}
\bigskip

The next result essentially contains \cite[Lemma 5.2]{PitP}, which relies on $%
|\varphi (0)|=1,$ the continuity of $\varphi ,$ and the existence of some $t$
with $\varphi (t)<1$ (guaranteed below by the non-triviality of $\varphi ).$
We assume less here, and so must also consider the possibility that $%
|\varphi (0)|=0.$

\bigskip

\begin{proposition}
\textit{If }$(ChFE)$\ \textit{is satisfied
by a non-trivial continuous function }$\varphi $\textit{\ and for some }%
$c>0,$\textit{\ }$|\varphi (t)|=$\textit{\ }$|\varphi (ct)|$\textit{\ for
all }$t>0,$\textit{\ then }$c=1$\textit{.}
\end{proposition}
\bigskip

\begin{proof} Note first that $a_{n}>0$ for all $n;$ indeed,
otherwise, for some $k\geq 1$%
\[
|\varphi (t)|^{k}=|\varphi (0)|\qquad (t\geq 0).
\]%
Assume first that $k>1;$ taking $t=0$ yields $|\varphi (0)|=0$ or $1,$ which
as in Lemma 2 implies $|\varphi |\equiv 0$ or $|\varphi |\equiv 1.$ If $k=1$
then $|\varphi (t)|=|\varphi (0)|$ and for all $n>1,$ $|\varphi
(0)|^{n}=|\varphi (0)|,$ so that again $|\varphi (0)|=0$ or $1,$ which again
implies $|\varphi |\equiv 0$ or $|\varphi |\equiv 1.$

Applying Lemma 2, the sequence $a_{n}$ converges either to $0$ or to $%
\infty .$

First suppose that $a_{n}\rightarrow 0.$ Then, as above (referring again to $%
K_{t}$), we obtain $|\varphi (t)|\leq 1$ for all $t.$ Now, since%
\[
|\varphi (0)|=\lim |\varphi (a_{n}t)|=\lim_{n}|\varphi (t)|^{n},
\]%
if $|\varphi (t)|=1$ for \textit{some} $t,$ then $|\varphi (0)|=1$, and that
in turn yields, for the very same reason, that $|\varphi (t)|\equiv 1$ for
\textit{all} $t,$ a trivial solution, which is ruled out. So in fact $%
|\varphi (t)|<1$ for all $t,$ and so $|\varphi (0)|=0.$ Now suppose that for
some $c>0,$ $|\varphi (t)|=$ $|\varphi (ct)|$ for all $t>0.$ We show that $%
c=1.$ If not, w.l.o.g. $c<1,$ (otherwise replace $c$ by $c^{-1}$ and note
that $|\varphi (t/c)|=$ $|\varphi (ct/c)|=|\varphi (t)|$ ); then%
\[
0=|\varphi (0)|=\lim_{n}|\varphi (c^{n}t)|=|\varphi (t)|,\text{ for }t>0,
\]%
and so $\varphi $ is trivial, a contradiction. So indeed $c=1$ in this case.

Now suppose that $a_{n}\rightarrow \infty .$ As $\varphi$ is non-trivial, choose $s$ with $\varphi
(s)\neq 0,$ then
\[
|\varphi (0)|=\lim_{n}|\varphi (s/a_{n})|=\lim_{n}\exp \left( \frac{1}{n}%
\log |\varphi (s)|\right) =1,
\]%
i.e. $|\varphi (0)|=1.$ Again suppose that for some $c>0,$ $%
|\varphi (t)|=$ $|\varphi (ct)|$ for all $t>0.$ To show that $c=1,$ suppose
again w.l.o.g. that $c<1;$ then%
\[
1=|\varphi (0)|=\lim_{n}|\varphi (c^{n}t)|=|\varphi (t)|\text{ for }t>0,
\]%
and so $|\varphi (t)|\equiv 1,$ again a trivial solution. So again $c=1.$%
\end{proof}
\bigskip

\textit{Proof of the Theorem. } $(ChFE)$ implies that%
\[
|\varphi (a_{mn}t)|=|\varphi (t)|^{mn}=|\varphi (a_{m}t)|^{n}=|\varphi
(a_{m}a_{n}t)|.
\]%
By Proposition 2, $a_n$ satisfies the discrete version of the Cauchy exponential equation $(CEE)$
\[
a_{mn}=a_{m}a_{n}\qquad (m,n\in \mathbb{N}),
\]%
whose solution is known to take the form $n^{k}$ (cf. \cite[Lemma 5.4]{PitP}),
since $a_{n}>0$ (as in Prop. 2). If $a_{n}=1$ for some $n>1,$ then, for each
$t>0,$ $|\varphi (t)|=0$ or 1  (as $|\varphi (t)|=|\varphi (t)|^{n}$) and so
again, by continuity as in Lemma 2, $\varphi $ is trivial. So $k\neq 0.$ $%
\square $

\begin{rem}
Continuity is essential to the theorem: take $a_n\equiv 1$, then a Borel function $\varphi$ may take the values 0 and 1 arbitrarily.
\end{rem}

\section{Complements}\label{s:complements}

\subsection{Self-neglecting and self-equivarying functions}\label{ss:SNSE}

Recall (cf.\ \cite[Section 2.11]{BinGT}) that a self-map $\varphi$ of
$\mathbb{R}_+$ is \emph{self-neglecting} ($\varphi\in{\it SN}$) if
\begin{equation}\label{SN}
\varphi(x+t\varphi(x))/\varphi(x)\rightarrow1\quad\text{locally uniformly in
$t$ for all $t\in\mathbb{R}_+$},\tag{\emph{SN}}
\end{equation}
and $\varphi(x)=\oh(x)$ as $x\rightarrow \infty$. This traditional restriction may be usefully relaxed
in two ways, as in \cite{Ost1}: firstly, in imposing the weaker order
condition
$\varphi(x)=\Oh(x)$, and secondly by replacing the limit $1$ by a general
limit function $\eta$, so that
\begin{equation}\label{SE}
\varphi(x+t\varphi(x))/\varphi(x)\rightarrow\eta(t)\quad\text{locally
uniformly in $t$ for all $t\in\mathbb{R}_+$}.\tag{\emph{SE}}
\end{equation}
A $\varphi $ satisfying \eqref{SE} is called \emph{self-equivarying} in
\cite{Ost1}, and the limit function $\eta=\eta^\varphi$ necessarily satisfies
the equation
\begin{equation}\label{BFE}
\eta(u+v\eta(u))=\eta(u)\eta(v)\qquad(u,v\in\mathbb{R}_+)
\tag{\emph{BFE}}
\end{equation}
(this is a special case of the \foreignlanguage{polish}{Go\l\aob b--Schinzel}
equation---see also e.g.\ \cite{Brz}, or \cite{BinO2}, where \eqref{BFE} is
termed the \emph{Beurling functional equation}). As $\eta\ge0$,
imposing the natural condition $\eta>0$ (on $\mathbb{R}_+$) implies that
it is continuous and of the form
$\eta(t)=1+\rho t$, for some $\rho\ge0$
(see \cite{BinO2}); the case $\rho=0$ recovers \eqref{SN}. A function
$\varphi\in{\it SE}$ has the representation
\[
\varphi(t)\sim\eta^\varphi(t)\int_1^t e(u)\sd u\quad\text{for some
continuous $e\rightarrow0$}
\]
(where $f\sim g$ if $f(x)/g(x)\rightarrow1$, as $x\rightarrow\infty$), and
the second factor is in ${\it SN}$ (see \cite[Theorem 9]{BinO1}, \cite{Ost1}).

\subsection{Theorem GFE}\label{ss:ThmGFE}
This theorem has antecedents in \cite{Acz} and \cite{Chu}, \cite[Theorem
  1]{Ost2}, and is generalized in \cite[Theorem 3]{BinO2}. It is also studied
in \cite{BinO3} and \cite{Ost2}.

\subsection{Homomorphisms and random walks}\label{ss:Homo}
In the context of a ring, the `\foreignlanguage{polish}{Go\l\aob b--Schinzel}
functions' $\eta_\rho(t)=1+\rho t$, as above, were used by Popa and Javor (see
\cite{Ost2} for references) to define associated (generalized) \emph{circle
  operations}:
\[
a\circ_\rho b=a+\eta_\rho(a)b=a+(1+\rho a)b=a+b+\rho ab.
\]
(Note that $a\circ_1b=a+b+ab$ is the familiar circle operation, and
$a\circ_0b=a+b$.) These were studied in the context of $\mathbb{R}$ in
\cite[Section 3.1]{Ost2}; it is straightforward to lift that analysis to the
present context of the ring $\mathbb{C}$, yielding the \emph{complex circle
  groups}
\[
\mathbb{C}_\rho:=\{x\in\mathbb{C}:1+\rho x\ne0\}
=\mathbb{C}\backslash\{\rho^{-1}\}\qquad (\rho\ne0).
\]
Since
\begin{align*}
(1+\rho a)(1+\rho b)&=1+\rho a+\rho b+\rho^2ab
  =1+\rho\lbrack a+b+\rho ab],\\
\eta_\rho(a)\eta_\rho(a)&=\eta_\rho(a\circ_\rho b),
\end{align*}
$\eta_\rho:(\mathbb{C}_\rho,\circ_\rho)\rightarrow
(\mathbb{C}^*,\cdot)=(\mathbb{C}\backslash\{0\},\times)$ is an isomorphism
(`from $\mathbb{C}_\rho$ to $\mathbb{C}_\infty$').

We may recast \eqref{GFEx} along the lines of \eqref{dag} so that
$G(s)=s^\gamma$ with $\gamma\ne0$, and
$K(t)=(t^\gamma-1)\rho^{-1}$, for
$$\rho=\frac\gamma\kappa=\frac{1-k}{k\kappa}.$$
Then, as $\eta_\rho(x)=1+\rho x=G(K^{-1}(x))$,
\[
K(st)=K(s)\circ_\rho K(t)=K(s)+\eta_\rho(K(s))K(t)=K(s)+G(s)K(t).
\]
For $\gamma\ne0$, $K$ is a homomorphism from the multiplicative reals
$\mathbb{R}_+$ into $\mathbb{C}_\rho$; more precisely, it is an
isomorphism between $\mathbb{R}_+$ and the conjugate subgroup
$(\mathbb{R}_+-1)\rho^{-1}$.
In the case $\gamma=0$ ($k=1$), $\mathbb{C}_0=\mathbb{C}$ is the additive
group of complex numbers; from \eqref{GFEx} it is immediate that $K$
maps logarithmically into $(\mathbb{R},+),$ `the additive reals'.

\acks The final form of this manuscript owes much to the referee's supportively
penetrating reading of an earlier draft, and to the editors' advice and good
offices, for which sincere thanks.

\bigskip

\section*{Appendix: a ratio formula}

We give an elementary derivation (using Riemann integrals) of the formula%
\[
\int_{0}^{\infty }\frac{\cos x}{x^{k}}e^{-\delta x}\,\mathrm{d}x\left/
\int_{0}^{\infty }\frac{\sin x}{x^{k}}e^{-\delta x}\,\mathrm{d}x\right.
=\tan \pi k/2\qquad (0<k<1).
\]%

Substitution for $\delta >0$ of $s=\delta +i$ $=re^{i\theta }$, with $%
r^{2}=1+\delta ^{2}$ and $\theta =\theta _{\delta }=\tan ^{-1}(1/\delta )$,
in the Gamma integral:%
\[
\frac{\Gamma (1-k)}{s^{1-k}}=\int_{0}^{\infty }\frac{e^{-sx}}{x^{k}}\,%
\mathrm{d}x,
\]%
with $0<k<1$, gives
\[
\int_{0}^{\infty }\frac{\cos x-i\sin x}{x^{k}}e^{-\delta x}\,\mathrm{d}x=%
\frac{\Gamma (1-k)}{(1+\delta ^{2})^{(1-k)/2}}[\cos (1-k)\theta _{\delta
}-i\sin (1-k)\theta _{\delta }]\qquad (\delta >0).
\]%
This yields in the limit as $\delta \downarrow 0,$ since $\theta _{\delta
}\rightarrow \pi /2,$ the ratio of the real and imaginary parts of the
left-hand side for $\delta =0$ to be%
\[
\cot (1-k)\pi /2=\tan \pi k/2.
\]%
Passage to the limit $\delta \downarrow 0$ on the left is validated, for any $k>0$,
by an appeal to Abel's method: first integration by parts (twice) yields an
indefinite integral%
\[
(1+\delta ^{2})\int e^{\delta x}\sin x\,\mathrm{d}x=-e^{\delta x}\cos
x+\delta e^{\delta x}\sin x,
\]%
valid for all $\delta ,$ whence (again by parts)%
\[
\int_{1}^{T}\frac{e^{-\delta x}\sin x\,\mathrm{d}x}{x^{k}}=\frac{e^{-\delta
}(\delta \sin 1+\cos 1)}{(1+\delta ^{2})}-\frac{e^{-\delta T}(\delta \sin
T+\cos T)}{T^{k}(1+\delta ^{2})}-k\int_{1}^{T}\frac{e^{-\delta x}(\delta
\sin x+\cos x)\,\mathrm{d}x}{x^{k+1}(1+\delta ^{2})}.
\]%
Here $e^{-\delta x}$ is uniformly bounded as $\delta \downarrow 0,$ so by
joint continuity on $[0,1]$
\begin{eqnarray*}
\lim_{\delta \downarrow 0}\int_{0}^{\infty }\frac{1}{x^{k}}e^{-\delta x}\sin
x\,\mathrm{d}x &=&\lim_{\delta \downarrow 0}\int_{0}^{1}\frac{1}{x^{k}}%
e^{-\delta x}\sin x\,\mathrm{d}x+\lim_{\delta \downarrow 0}\int_{1}^{\infty }%
\frac{1}{x^{k}}e^{-\delta x}\sin x\,\mathrm{d}x \\
&=&\int_{0}^{\infty }\frac{\sin x}{x^{k}}\,\mathrm{d}x,
\end{eqnarray*}%
and likewise with $\cos $ for $\sin $.

\end{document}